\documentclass[11pt]{amsart}
\usepackage[foot]{amsaddr}

\makeatletter
\renewcommand{\email}[2][]{%
  \ifx\emails\@empty\relax\else{\g@addto@macro\emails{,\space}}\fi%
  \@ifnotempty{#1}{\g@addto@macro\emails{\textrm{(#1)}\space}}%
  \g@addto@macro\emails{#2}%
}
\makeatother

\usepackage{amssymb}
\usepackage{amsmath}
\usepackage{amsfonts}
\usepackage{mathrsfs}
\usepackage{psfrag}
\usepackage{verbatim}
\usepackage{color}
\usepackage{bm}
\usepackage[usenames,dvipsnames]{xcolor}

\DeclareGraphicsExtensions{}

\theoremstyle{plain}
\newtheorem{theorem}{Theorem}[section]

\newtheorem{lemma}[theorem]{Lemma}
\newtheorem{proposition}[theorem]{Proposition}

\theoremstyle{definition}
\newtheorem{assumption}[theorem]{Assumption}

\theoremstyle{remark}
\newtheorem{remark}[theorem]{Remark}

\newcommand{\eps}{\varepsilon}

\newcommand{\la}{\left\langle}
\newcommand{\ra}{\right\rangle}

\newcommand{\BR}{\mathbb{R}}
\newcommand{\BP}{\mathbb{P}}
\newcommand{\BE}{\mathbb{E}}
\newcommand{\BZ}{\mathbb{Z}}
\newcommand{\BN}{\mathbb{N}}

\newcommand{\filt}{\mathscr{F}}

\newcommand{\normaltail}{\mathscr{T}}

\newcommand{\sfmin}{\mathsf{min}}

\newcommand{\sfX}{\mathsf{X}}
\newcommand{\sfx}{\mathsf{x}}

\newcommand{\aon}{\textsf{on}}
\newcommand{\aoff}{\textsf{off}}
\newcommand{\refer}{\textsf{ref}}
\newcommand{\border}{\textsf{border}}

\newcommand{\Tfin}{\mathsf{T}}

\newcommand{\sfz}{\mathsf{z}}

\begin{document}
\title{Randomly perturbed switching dynamics of a dc/dc converter}
\author{Chetan D. Pahlajani}
\address{Discipline of Mathematics\\ Indian Institute of Technology Gandhinagar\\ Palaj, Gandhinagar 382355\\ India}
\email{cdpahlajani@iitgn.ac.in} 
\date{\today.}

%\email{cdpahlajani@iitgn.ac.in} 
%\date{\today.}

\begin{abstract}
In this paper, we study the effect of small Brownian noise on a switching dynamical system which models a first-order {\sc dc}/{\sc dc} buck converter. The state vector of this system comprises a continuous component whose dynamics switch, based on the {\sc on}/{\sc off} configuration of the circuit, between two ordinary differential equations ({\sc ode}), and a discrete component which keeps track of the {\sc on}/{\sc off} configurations. Assuming that the parameters and initial conditions of the unperturbed system have been tuned to yield a stable periodic orbit, we study the stochastic dynamics of this system when the forcing input in the {\sc on} state is subject to small white noise fluctuations of size $\eps$, $0<\eps \ll 1$. For the ensuing stochastic system whose dynamics switch at random times between a small noise stochastic differential equation ({\sc sde}) and an {\sc ode}, we prove a functional law of large numbers which states that in the limit of vanishing noise, the stochastic system converges to the underlying deterministic one on time horizons of order $\mathscr{O}(1/\eps^\nu)$, $0 \le \nu < 2/3$.
\end{abstract}

%\begin{keyword}
%stochastic differential equation \sep switching dynamical system \sep buck converter 
%\end{keyword}

\maketitle

\color{black}
\section{Introduction}\label{S:Intro}
Ordinary differential equations ({\sc ode}) and dynamical systems play a fundamental role in modelling and analysis of various phenomena arising in science and engineering. In many applications, however, the smooth evolution of the {\sc ode} dynamics is punctuated by discrete instantaneous events which give rise to switching or non-smooth behaviour. Examples include instantaneous switching between different governing {\sc ode} in a power electronic circuit \cite{BKYY,BV_PowerElectronics,dBGGV}, discontinuous change in velocity for an oscillator impacting a boundary \cite{SH83,Nor91}, etc. In such instances, the dynamical system involves functions which are not smooth, but only piecewise-smooth in their arguments. Such piecewise-smooth dynamical systems \cite{dBBCK} display a wealth of phenomena not seen in their smooth counterparts, and have hence been the subject of much current research. 

Dynamical systems arising in practice are almost always subject to random disturbances, owing perhaps to fluctuating external forces, or uncertainties in the system, or unmodelled dynamics, etc. A more accurate picture can therefore be obtained by modelling such systems (at least in the continuous-time case) using {\it stochastic differential equations} ({\sc sde}); intuitively, this corresponds to adding a ``noise" term to the {\sc ode}. For cases where the perturbing noise is small, it is natural to ask whether the stochastic (perturbed) system converges to the deterministic (unperturbed) one in the limit of vanishing noise, and if yes, how the asymptotic behaviour of the fluctuations may be quantified. Such questions have played a significant role in the development of limit theorems for stochastic processes; see, for instance \cite{DZ98,EK86,FW_RPDS,PS_Multiscale}.

%especially the work of Freidlin and Wentzell \cite{FW_RPDS} on large deviations and averaging for randomly perturbed dynamical systems. 

%The Freidlin-Wentzell theory enables one to quantify the asymptotic behaviour, in the limit of vanishing noise, of the deviation of the stochastic system from the underlying deterministic one. The results parallel the classical limit theorems of probability: a Law of Large Numbers ({\sc lln}) for the limiting mean behaviour of the stochastic system, a Central Limit Theorem ({\sc clt}) describing the limiting fluctuations about the mean, a Large Deviations theory for estimating probabilities of rare events, and so on.
%or macroscopic random transitions due to the cumulative effect of noise. 

Although smooth dynamical systems perturbed by noise have been analysed in great depth over the past few decades, the effect of random noise on non-smooth or switching dynamical systems remains, with some exceptions (see, for instance, \cite{CL_TAC_2007,CL_SICON,HBB1,HBB2,SK14,SK_SD,SK_JNS}), relatively unexplored. One of the challenges in such an undertaking is that even in the absence of noise, the dynamics of switching systems can prove rather difficult to analyse. Part of the reason is the frequently encountered intractability of such systems to analytic computation \cite{BC_Boost, dBGGV}, even in cases when the component subsystems are linear. 

Our primary interest is the study of stochastic processes which arise due to small random perturbations of (non-smooth) switching dynamical systems. These problems are of immediate relevance in the analysis of {\sc dc/dc} converters in power electronics---naturally susceptible to noise---in which time-varying circuit topology leads to mathematical models characterised by switching between different governing {\sc ode}. In the purely deterministic setting, the dynamics of these systems have been extensively studied, with much of the work focussing on {\it buck converters} \cite{BKYY,dBGGV,HDJ92,FO96}; these are circuits used to transform an input {\sc dc} voltage to a lower output {\sc dc} voltage. Perhaps the simplest of these is the first-order buck converter; this is a system which switches between two linear first-order {\sc ode}. While this circuit is pleasantly amenable to some explicit computation, it nevertheless displays rich dynamics in certain parameter regimes. Periodic orbits, bifurcations and chaos for this converter have been studied in \cite{BKYY,HDJ92}. 

In the present paper, we study small random perturbations of a switching dynamical system which models a first-order buck converter. The state vector of this system comprises a continuous component (the inductor current) governed by one of two different {\sc ode}, and a discrete component which takes values $1$ or $0$ depending on whether the circuit is in the {\sc on} versus {\sc off} configurations, respectively. Assuming that the parameters and initial conditions of the unperturbed system have been tuned to yield a stable periodic orbit, we study the stochastic dynamics of this system when the forcing (input {\sc dc} voltage) is subject to small white noise fluctuations of size $\eps$, with $0< \eps \ll 1$. Our main result is a {\it functional law of large numbers} ({\sc flln}) which states that, as $\eps \searrow 0$, the solution of the stochastically perturbed system converges to that of the underlying deterministic system, over time horizons $\Tfin_\eps$ of order $\mathscr{O}(1/\eps^\nu)$ for any $0 \le \nu < 2/3$.

Part of the novelty of this work, in the context of the literature on switching diffusions (see, e.g., \cite{BBG_JMAA_1999,LuoMao,YZ10,YZ_book}), is that the switching in our problem is {\it not} driven by a discrete-state stochastic process (whose transitions may occur at a rate depending on the continuous component of the state); rather, the switching occurs whenever the continuous component of the state hits a threshold ({\sc on} $\to$ {\sc off}), or upon the arrival of a time-periodic signal ({\sc off} $\to$ {\sc on}). Our switching is thus {\it entirely} determined by the continuous component, together with a periodic clock signal. We also note that since the input {\sc dc} voltage in the buck converter influences the inductor current only in the {\sc on} state \cite{BKYY}, the perturbed system has alternating stochastic and deterministic evolutions: the dynamics switch {\it at random times} between an {\sc sde} driven by a small Brownian motion of size $\eps$ in the {\sc on} state, and an {\sc ode} in the {\sc off} state. The import of our results is that even in the presence of small stochastic perturbations, one may expect the buck converter to function close to its desired operation for ``reasonably long" times. 

The rest of the paper is organised as follows. In Section \ref{S:ProblemStatement}, we describe the switching systems (deterministic and stochastic) in some detail, and we pose our problem of interest. Next, in Section \ref{S:MainResult}, we state our main result (Theorem \ref{T:FLLN}) and outline the steps to the proof through a sequence of auxiliary lemmas and propositions. A few of these auxiliary results are proved in Section \ref{S:MainResult}, with the remainder (the slightly lengthier ones) being deferred to Section \ref{S:Proofs}.

\section{Problem Description}\label{S:ProblemStatement}
In this section, we formulate our problem of interest. We start with a description, including the governing {\sc ode}'s and the switching mechanism,  of a dynamical system modelling a first-order buck converter in Section \ref{SS:Det_sw_sys}. Random perturbations of this system, which lead to a switching {\sc sde}/{\sc ode} model, are discussed in Section \ref{SS:Stoch_sw_sys}. In Section \ref{SS:explicit_formulas}, we obtain explicit formulas for solutions to both the {\sc sde} and the {\sc ode}'s {\it between} switching times, and we piece these together {\it at} switching times to obtain expressions describing the overall evolution of both the perturbed (stochastic) and unperturbed (deterministic) switching systems. Finally, after showing in Section \ref{SS:stable_periodic_orbit} how problem parameters can be tuned and initial conditions chosen to ensure that the unperturbed system has a stable periodic orbit, we pose our questions of interest.

Before proceeding further, we note that we have a {\it hybrid} system. Indeed, the full state of the system is specified by a vector $z=(x,y)$ taking values in $\mathscr{Z} \triangleq \BR \times \{0,1\}$; here, $x \in \BR$ is the continuous component of the state---corresponding to the inductor current in the buck converter---while the discrete component $y$ takes values $1$ or $0$ depending on whether the switch is {\sc on} or {\sc off}. 

\subsection{Deterministic switching system}\label{SS:Det_sw_sys}
As noted above, the state of our system at time $t \in [0,\infty)$ will be specified by a vector $z(t) \triangleq (x(t),y(t))$ taking values in $\mathscr{Z} \triangleq \BR \times \{0,1\}$. We will assume that the dynamics of $x(t)$ when $y(t)=1$ ({\sc on} configuration) are governed by the {\sc ode}
\begin{equation}\label{E:ode-on}
\frac{dx}{dt}= - \alpha_\aon \thinspace x + \beta,
\end{equation}
while the dynamics of $x(t)$ when $y(t)=0$ ({\sc off} configuration) are described by
\begin{equation}\label{E:ode-off}
\frac{dx}{dt}= - \alpha_\aoff  \thinspace x. 
\end{equation}
Here, $\beta$, $\alpha_\aon$, $\alpha_\aoff$ are fixed positive parameters with $\beta$ representing the (rescaled) input voltage of an external power source, while $\alpha_\aon$ and $\alpha_\aoff$ denote the (rescaled) resistances in the {\sc on} and {\sc off} configurations, respectively.\footnote{More precisely, $\beta = V_{\mathsf{in}}/L$, $\alpha_\aon=R/L$ and $\alpha_\aoff=(R+r_d)/L$, where $V_{\mathsf{in}}$ is the input voltage, $R$ denotes the load resistance, and $r_d$ is the diode resistance \cite{BKYY}.}

 The switching between the {\sc on} and {\sc off} configurations is effected as follows. A reference level $x_\refer  \in \left(0,\beta/\alpha_\aon \right)$ is fixed. Suppose the system starts in the {\sc on} configuration, i.e., $y(0)=1$, with $x(0) \in (0,x_\refer)$. The current $x(t)$ increases according to \eqref{E:ode-on}, with $y(t)$ staying at 1, until $x(t)$ hits the level $x_{\sf ref}$. At this point, an {\sc on} $\to$ {\sc off} transition occurs: $y(t)$ jumps to 0 and $x(t)$ now evolves according to \eqref{E:ode-off}. This continues until the next arrival of a periodic clock signal with period 1 (which arrives at times $n \in \BN$) triggers an {\sc off} $\to$ {\sc on} transition: $y(t)$ jumps back to 1, $x(t)$ again evolves according to \eqref{E:ode-on}, and the cycle continues. Note that if a clock pulse arrives in the {\sc on} configuration, it is ignored. Of course, if one starts in the {\sc off} configuration, $x(t)$ evolves according to \eqref{E:ode-off} until the next clock pulse, at which point the system goes {\sc on}, and the subsequent dynamics are as described above. An important assumption in our analysis is that $x(t)$ is continuous across switching times.

\subsection{Random perturbations}\label{SS:Stoch_sw_sys}

We now suppose that the forcing term $\beta$ in \eqref{E:ode-on} is subjected to small white noise perturbations of size $\eps$, $0 < \eps \ll 1$; for the buck converter, this corresponds to small random fluctuations in the input voltage. In this setting, the state of the system at time $t \in [0,\infty)$ is given by a stochastic process $Z^\eps_t \triangleq (X^\eps_t,Y^\eps_t)$ taking values in $\mathscr{Z}$. The dynamics of $X^\eps_t$ in the {\sc on} configuration ($Y^\eps_t=1$) are now governed by the {\sc sde}
\begin{equation}\label{E:sde-on}
dX^\eps_t = (-\alpha_\aon X^\eps_t + \beta)dt + \eps dW_t,
\end{equation}
where $W_t$ is a standard one-dimensional Brownian motion, while evolution of $X^\eps_t$ in the {\sc off} state ($Y^\eps_t=0$) is governed by the {\sc ode} \eqref{E:ode-off}, as before. The switching mechanism is similar to that in the unperturbed case, but with the stochastic processes $X^\eps_t$, $Y^\eps_t$ playing the roles of $x(t)$, $y(t)$. Note, in particular, that the times for {\sc on} $\to$ {\sc off} transitions are given by {\it passage times} of $X^\eps_t$ (governed by \eqref{E:sde-on}) to the level $x_\refer$. As before, $X^\eps_t$ is assumed to be continuous across switching times.

\subsection{Explicit formulas}\label{SS:explicit_formulas}

The foregoing discussion makes clear how $z(t)=(x(t),y(t))$ and $Z^\eps_t = (X^\eps_t,Y^\eps_t)$ are to be obtained, once an initial condition $z_0=(x_0,y_0) \in \mathscr{Z}$ has been specified: the evolutions of $x(t)$ and $X^\eps_t$ are given, respectively, by {\it concatenating} solutions to \eqref{E:ode-on} and \eqref{E:ode-off}, and solutions to \eqref{E:sde-on} and \eqref{E:ode-off}, at the respective switching times, maintaining continuity. The function $y(t)$ and the sample paths of $Y^\eps_t$---which are piecewise constant and take values in $\{0,1\}$---will be assumed to be right-continuous. Below, we obtain expressions for $z(t)$ and $Z^\eps_t$ starting from initial condition $z_0=(x_0,1)$ where $x_0 \in (0,x_\refer)$. We note that starting with $y_0=1$ entails no real loss of generality; indeed, as will become apparent, the expressions below can be easily modified to accommodate the case when $y_0=0$. 

In the sequel, we will use $1_A$ to denote the indicator function of the set $A$, and for real numbers $a,b$, we let $a \wedge b$ and $a \vee b$ denote the minimum and maximum of $a$ and $b$, respectively.

%Also, in Assumption \ref{A:assumptions_real}, we will describe how the parameters $\beta$, $\alpha_\aon$, $\alpha_\aoff$ and $x_\refer$ can be tuned to ensure that the system contains a stable periodic orbit. Our primary interest is to understand the asymptotic behaviour of $Z^\eps=(X^\eps,Y^\eps)$ as $\eps \searrow 0$. 

\subsubsection{Solution of deterministic switching system}
As indicated above, we fix an initial condition $z_0=(x_0,1)$ with $x_0 \in (0,x_\refer)$. Let $s_0 \triangleq 0$, and set $\sfx^{0,\aoff}(t) \equiv x_0$. Next, define
%\begin{equation*}
$\sfx^{1,\aon}(t) \triangleq 1_{[s_0,\infty)}(t) \cdot \left\{\beta/\alpha_\aon + \left(x_0 - \beta/\alpha_\aon \right) e^{-\alpha_\aon t}\right\}$ for $t \ge 0$.
%\end{equation*}
Let
%\begin{equation*}
$t_1 \triangleq \inf\{t > 0:\sfx^{1,\aon}(t)=x_\refer\}$
%\end{equation*} 
be the first time that $\sfx^{1,\aon}(t)$ reaches level $x_\refer$ and define
%\begin{equation*}
$\sfx^{1,\aoff}(t) \triangleq 1_{[t_1,\infty)}(t) \cdot x_\refer \medspace e^{-\alpha_\aoff(t-t_1)}.$
%\end{equation*}
Let $s_1 \triangleq \inf \{t>t_1:t \in \BZ\}$ be the time of arrival of the next clock pulse. The solution of the deterministic switching system on the interval $[s_0,s_1)$ is now given by
%\begin{equation*}
$\sfx^1(t) \triangleq  \sfx^{1,\aon}(t) \cdot 1_{[s_0,t_1)}(t)+ \sfx^{1,\aoff}(t) \cdot 1_{[t_1,s_1)}(t)$.
%\end{equation*}

In general, given the solution over $[s_0,s_{n-1})$, the solution $\sfx^{n}(t)$ over $[s_{n-1},s_n)$ is obtained as follows. We let
\begin{equation}\label{E:det_state_pieces}
\begin{aligned}
\sfx^{n,\aon}(t) &\triangleq 1_{[s_{n-1},\infty)}(t) \cdot \left\{\frac{\beta}{\alpha_\aon} + \left(\sfx^{n-1,\aoff}(s_{n-1}) - \frac{\beta}{\alpha_\aon}\right) e^{-\alpha_\aon (t-s_{n-1})}\right\} \qquad \text{for $t \ge 0$,}\\
t_{n} &\triangleq \inf\{t>s_{n-1}:\sfx^{n,\aon}(t)=x_\refer\},\\
\sfx^{n,\aoff}(t) &\triangleq 1_{[t_{n},\infty)}(t) \cdot x_\refer \medspace e^{-\alpha_\aoff(t-t_{n})} \qquad \text{for $t \ge 0$,}\\
s_{n} &\triangleq \inf\{t>t_{n}: t \in \BZ\},\\
\sfx^{n}(t) &\triangleq \sfx^{n,\aon}(t) \cdot 1_{[s_{n-1},t_n)}(t) + \sfx^{n,\aoff}(t) \cdot 1_{[t_{n},s_{n})}(t).
\end{aligned}
\end{equation}
The evolution of the deterministic switching system over $[0,\infty)$ is now given by 
\begin{equation}\label{E:det_state_full}
z(t) =(x(t),y(t)) \qquad \text{where} \quad x(t) \triangleq  \sum_{n \ge 1} \sfx^{n}(t), \quad y(t) \triangleq \sum_{n \ge 1} 1_{[s_{n-1},t_n)}(t). 
\end{equation}

We have thus decomposed the evolution into a sequence of {\sc on}/{\sc off} cycles with the switching times $t_n$ and $s_n$ corresponding to the $n$-th {\sc on} $\to$ {\sc off} and {\sc off} $\to$ {\sc on} transitions, respectively; next, we have solved the {\sc ode} \eqref{E:ode-on} and \eqref{E:ode-off} between switching times,  and then linked the pieces together while maintaining continuity of $x(t)$ at switching times. 
%If we would like to also specify the initial condition, we will write $x(t;x_0,y_0)$ and $y(t;x_0,y_0)$ respectively.

\subsubsection{Solution of stochastic switching system}
We now provide a similar detailed construction of the stochastic process $Z^\eps_t=(X^\eps_t,Y^\eps_t)$ starting from the same initial condition $z_0=(x_0,1)$. Let $W=\{W_t, \filt_t:0 \le t < \infty\}$ be a standard one-dimensional Brownian motion on the probability space $(\Omega,\filt,\BP)$. We introduce, for each $n \in \BN$, the processes $\sfX^{n,\eps,\aon}_t$, $\sfX^{n,\eps,\aoff}_t$, $\sfX^{n,\eps}_t$ and random switching times $\tau_n^\eps$, $\sigma_n^\eps$, which are defined recursively as follows. Set $\sigma_0^\eps \triangleq 0$ and define $\sfX^{0,\eps,\aoff}_t \equiv x_0$. Now, let 
%\begin{equation*}
$\sfX^{1,\eps,\aon}_t \triangleq 1_{[\sigma_0^\eps,\infty)}(t) \cdot \left\{\beta/\alpha_\aon + \left(x_0- \beta/\alpha_\aon \right) e^{-\alpha_\aon t} + \eps \int_{0}^t e^{-\alpha_\aon(t-u)} dW_u\right\}$ for $t \ge 0$,
%\end{equation*}
and let $\tau_1^\eps \triangleq \inf\{t >0:\sfX^{1,\eps,\aon}_t=x_\refer  \}$ be the first passage time of $\sfX^{1,\eps,\aon}_t$ to level $x_\refer$. We next define
%\begin{equation*}
$\sfX^{1,\eps,\aoff}_t \triangleq 1_{[\tau_1^\eps,\infty)}(t) \cdot x_\refer \medspace e^{-\alpha_\aoff(t-\tau_1^\eps)}$ 
%\end{equation*}
and let $\sigma_1^\eps \triangleq \inf\{t> \tau_1^\eps:t \in \BZ\}$ be the time of arrival of the next clock pulse. We now set
%\begin{equation*}
$\sfX^{1,\eps}_t \triangleq \sfX^{1,\eps,\aon}_t \cdot 1_{[\sigma_0^\eps,\tau_1^\eps)}(t) + \sfX^{1,\eps,\aoff}_t \cdot 1_{[\tau_1^\eps,\sigma_1^\eps)}(t)$.
%\end{equation*}

To compactly express $\sfX^{n,\eps,\aon}_t$ for general $n$, let $I=\{I_t:0 \le t < \infty\}$ be the process defined by
%\begin{equation}\label{E:I_def}
$I_t \triangleq \int_0^t e^{\alpha_\aon u} dW_u$ for $t \ge 0$.
%\end{equation}
Note that $I$ is a continuous, square-integrable Gaussian martingale. We now define, for each $n \in \BN$, 
\begin{equation}\label{E:stoch_state_pieces}
\begin{aligned}
\sfX^{n,\eps,\aon}_t &\triangleq 1_{[\sigma_{n-1}^\eps,\infty)}(t) \cdot \left\{  \frac{\beta}{\alpha_\aon} + \left(\sfX^{n-1,\eps,\aoff}_{\sigma_{n-1}^\eps}- \frac{\beta}{\alpha_\aon}\right) e^{-\alpha_\aon (t-\sigma_{n-1}^\eps)} + \eps \thinspace e^{-\alpha_\aon t} \left(I_t - I_{t\wedge \sigma_{n-1}^\eps}  \right) \right\},\\
\tau_{n}^\eps &\triangleq \inf\{t>\sigma_{n-1}^\eps: \sfX^{n,\eps,\aon}_t = x_\refer\},\\
\sfX^{n,\eps,\aoff}_t &\triangleq 1_{[\tau_{n}^\eps,\infty)}(t) \cdot x_\refer \medspace e^{-\alpha_\aoff(t-\tau_{n}^\eps)},\\
\sigma_{n}^\eps &\triangleq \inf\{t > \tau_{n}^\eps: t \in \BZ\},\\
\sfX^{n,\eps}_t &\triangleq  \sfX^{n,\eps,\aon}_t \cdot 1_{[\sigma_{n-1}^\eps,\tau_{n}^\eps)}(t) + \sfX^{n,\eps,\aoff}_t \cdot 1_{[\tau_{n}^\eps,\sigma_{n}^\eps)}(t).
\end{aligned}
\end{equation}
%where the stochastic integral in the first equation of \eqref{E:stoch_state_pieces} is defined by
%%\begin{equation*}
%$\int_{\sigma_{n-1}^\eps}^t e^{-\alpha_\aon(t-u)} dW_u = \int_0^t 1_{\{u>\sigma_{n-1}^\eps\}} \cdot e^{-\alpha_\aon(t-u)} dW_u$.
%%\end{equation*}
Our stochastic process of interest is now given by
\begin{equation}\label{E:stoch_state_full}
Z^\eps_t \triangleq (X^\eps_t,Y^\eps_t), \qquad \text{where} \qquad X^\eps_t \triangleq \sum_{n \ge 1} \sfX^{n,\eps}_t, \quad Y^\eps_t \triangleq  \sum_{n \ge 1} 1_{[\sigma_{n-1}^\eps,\tau_n^\eps)}(t). 
\end{equation}

Once again, the evolution comprises a sequence of {\sc on}/{\sc off} cycles, with the quantities above admitting a natural interpretation which parallels the unperturbed (deterministic) case.

%\begin{remark}\label{R:I_def}
%Since 
%%\begin{equation*}
%$\int_0^t 1_{\{u>\sigma_{n-1}^\eps\}} \cdot e^{-\alpha_\aon(t-u)} dW_u = e^{-\alpha_\aon t} (I_t - I_{t\wedge \sigma_{n-1}^\eps} )$
%%\end{equation*}
%(see \cite[Proposition 3.2.10]{KS91}),   
%$\sfX^{n,\eps,\aon}_t$ in the first equation of \eqref{E:stoch_state_pieces} can be written
%\begin{equation}\label{E:stoch_on_alternate}
%\end{equation}
%\end{remark}

\subsection{Stable periodic orbit}\label{SS:stable_periodic_orbit}
We now describe the assumptions on the problem parameters that ensure the existence of a stable periodic solution to \eqref{E:det_state_pieces}, \eqref{E:det_state_full}. The argument proceeds by analysing the {\it stroboscopic} map \cite{BKYY} which takes the system state at one clock instant to the state at the next. Map-based techniques are used extensively in analysing the switching dynamics of power electronic circuits; see also \cite{dBGGV,HDJ92}. 

%\begin{assumption}\label{A:assumptions}
%Fix $x_\refer>0$, $\alpha_\aon>0$. Select $\beta>0$ such that
%\begin{equation}\label{E:beta_assumption}
%x_\refer \thinspace \alpha_{\aon} < \beta < \left( \frac{e^{\alpha_\aon}}{e^{\alpha_\aon}-1}\right) x_\refer \thinspace \alpha_\aon.
%\end{equation}
%Let $\alpha_\aoff>0$ such that
%\begin{equation}\label{E:alpha_off_assumption}
%\alpha_\aoff < \left(\frac{\beta/\alpha_\aon - x_\refer}{x_\refer}\right) \alpha_\aon.
%\end{equation} 
%\end{assumption}

\begin{assumption}\label{A:assumptions_real}
Fix $x_\refer>0$, $0< \alpha_\aon < \log 2$. Select $\beta>0$ such that
\begin{equation}\label{E:beta_assumption_real}
2 x_\refer \thinspace \alpha_{\aon} < \beta < \left( \frac{e^{\alpha_\aon}}{e^{\alpha_\aon}-1}\right) x_\refer \thinspace \alpha_\aon.
\end{equation}
Let $\alpha_\aoff>0$ such that
\begin{equation}\label{E:alpha_off_assumption_real}
\alpha_\aon < \alpha_\aoff < \left(\frac{\beta/\alpha_\aon - x_\refer}{x_\refer}\right) \alpha_\aon.
\end{equation} 
\end{assumption}

%Before proceeding, a few comments are in order. First, Proposition \ref{P:fixed_point} stated below is valid under either of Assumptions \ref{A:assumptions} and \ref{A:assumptions_real}. However, if---as in the case of the first-order buck converter---one would like $\alpha_\aoff>\alpha_\aon$, then one needs to ensure that the right-hand side in \eqref{E:alpha_off_assumption} is greater than $\alpha_\aon$, which in turn necessitates the lower bound on $\beta$ in \eqref{E:beta_assumption_real}. The latter equation now requires that $e^{\alpha_\aon}/(e^{\alpha_\aon}-1)>2$, forcing $\alpha_\aon<\log 2$.\footnote{The upper bound on $\beta$ in \eqref{E:beta_assumption} or \eqref{E:beta_assumption_real} ensures that $x_\border>0$, while the upper bound on $\alpha_\aoff$ in \eqref{E:alpha_off_assumption} or \eqref{E:alpha_off_assumption_real} ensures the stability of $x^*$.}  

We now define a map $f:[0,x_\refer] \to [0,x_\refer]$ which maps $x_0 \in [0,x_\refer]$ to the solution $x(t)$ at time 1, subject to the initial condition being $(x_0,1)$, i.e., $f:x_0 \mapsto x(1;x_0,1)$. We are interested in the case when $f$ is only {\it piecewise smooth}. Put another way, if we let 
%\begin{equation*}
$x_\border \triangleq \beta/\alpha_\aon + \left( x_\refer - \beta/\alpha_\aon   \right) e^{\alpha_\aon}$
%\end{equation*}
be the particular value of $x_0$ for which the corresponding $\sfx^{1,\aon}(t)$ satisfies $\sfx^{1,\aon}(1)=x_\refer$, we would like $x_\border \in (0,x_\refer)$. It is easily seen that the upper bound on $\beta$ in \eqref{E:beta_assumption_real} ensures that such is indeed the case. The map $f:x_0 \mapsto x(1;x_0,1)$ is seen to be given by 
\begin{equation}\label{E:psmap} 
f(x) \triangleq \begin{cases} \frac{\beta}{\alpha_\aon} + \left( x-\frac{\beta}{\alpha_\aon} \right) e^{-\alpha_\aon} \qquad & \text{if $0 \le x \le x_\border$,}\\ x_\refer \thinspace e^{-\alpha_\aoff} \left( \frac{\beta/\alpha_\aon - x}{\beta/\alpha_\aon - x_\refer}  \right)^{\alpha_\aoff/\alpha_\aon} \qquad & \text{if $x_\border<x\le x_\refer$.} \end{cases}
\end{equation}

\begin{proposition}\label{P:fixed_point}
Suppose Assumption \ref{A:assumptions_real} holds. Then, the mapping $f:[0,x_\refer] \to [0,x_\refer]$ has a unique fixed point $x^*$ which lies in the interval $(x_\border,x_\refer)$. Further, $|f^\prime(x^*)|<1$, implying that $x^*$ is a stable fixed point of the discrete-time dynamical system $x_n \mapsto f(x_n)$.
\end{proposition}

\begin{proof}
Let $h(x) \triangleq f(x)-x$. Note that $h(0)=f(0)>0$, $h(x_\refer)=f(x_\refer)-x_\refer=x_\refer(e^{-\alpha_\aoff}-1)<0$. Since $h$ is continuous, there exists $x^* \in (0,x_\refer)$ such that $h(x^*)=0$, i.e., $f(x^*)=x^*$. Since $f(x)>x$ for all $x \in [0,x_\border]$, we must have $x^* \in (x_\border,x_\refer)$. Further, since $f(x)$ decreases on $(x_\border,x_\refer)$ as $x$ increases over this same interval, $f$ can have at most one fixed point. It is easily checked that for $x \in (x_\border,x_\refer)$, we have
\begin{equation*}
|f^\prime(x)| =  \frac{\alpha_\aoff \thinspace f(x)}{\beta - \alpha_\aon \thinspace x} \le \frac{\alpha_\aoff \thinspace x_\refer}{\beta - \alpha_\aon \thinspace x_\refer} < 1,
\end{equation*}
where the last inequality follows from the upper bound on $\alpha_\aoff$ in \eqref{E:alpha_off_assumption_real}. This proves stability of $x^*$.
\end{proof}

We can now pose our principal questions of interest. Suppose  $z(\cdot)$, $Z^\eps_\cdot$ are obtained from  \eqref{E:det_state_full}, \eqref{E:stoch_state_full}, respectively, with initial conditions $z(0)=Z^\eps_0=(x^*,1)$, where $x^*$ is as in Proposition \ref{P:fixed_point}. 
\begin{itemize}
\item For any fixed $\Tfin \in \BN$, do the dynamics of $Z^\eps_\cdot$ converge to those of $z(\cdot)$ in a suitable sense as $\eps \searrow 0$?
\item If yes, can the results be strengthened to the case when $\Tfin=\Tfin_\eps \in \BN$ grows to infinity, but ``not too fast", as $\eps \searrow 0$?
\end{itemize} 
In the next section, we will show that both these questions can be answered in the affirmative, provided $\Tfin_\eps = \mathscr{O}(1/\eps^\nu)$ with $0 \le \nu < 2/3$.
%\newpage

\section{Main Result}\label{S:MainResult}
Recall that the state space for the evolution of $z(t)$ and $Z^\eps_t$ is $\mathscr{Z}=\BR \times \{0,1\}$, which inherits the metric 
%\begin{equation*}
$r(z_1,z_2) \triangleq \left\{|x_1-x_2|^2 + |y_1-y_2|^2 \right\}^{1/2}$ for $z_1=(x_1,y_1)$, $z_2=(x_2,y_2) \in \mathscr{Z}$,
%\end{equation*}
from $\BR^2$. If $I$ is a closed subinterval of $[0,\infty)$, we let $D(I;\mathscr{Z})$ be the space of functions $\sfz:I \to \mathscr{Z}$ which are right-continuous with left limits. This space can be equipped with the Skorokhod metric $d_I$ \cite{ConvProbMeas,EK86}, which renders it complete and separable. If $\sfz \in D(I;\mathscr{Z})$ and $J$ is a closed subinterval of $I$, then the restriction of $\sfz$ to $J$ is an element of $D(J;\mathscr{Z})$ which, for simplicity of notation, will also be denoted by $\sfz$. For our switching systems of interest, we note that the function $z(t)$ in \eqref{E:det_state_full}, and the sample paths of the process $Z^\eps_t$ in \eqref{E:stoch_state_full}, belong to $D([0,\infty);\mathscr{Z})$. Our goal here is to study the convergence, as $\eps \searrow 0$, of $Z^\eps_\cdot$ to $z(\cdot)$ in the space $D([0,\Tfin_\eps];\mathscr{Z})$ for time horizons $\Tfin_\eps = \mathscr{O}(1/\eps^\nu)$ where $0 \le \nu <2/3$.

We start by defining the Skorokhod metric $d_I$ on the space $D(I;\mathscr{Z})$, where $I=[0,\Tfin]$ for some $\Tfin >0$.\footnote{See \cite{ConvProbMeas,EK86} for the case $I=[0,\infty)$.} Let $\tilde\Lambda_\Tfin$ be the set of all strictly increasing continuous mappings from $[0,\Tfin]$ onto itself,\footnote{Thus, we have $\lambda(0)=0$ and $\lambda(\Tfin)=\Tfin$ for all $\lambda \in \tilde{\Lambda}_\Tfin$.} and let $\Lambda_\Tfin$ be the set of functions $\lambda \in \tilde{\Lambda}_\Tfin$ for which  
\begin{equation*}
\gamma_\Tfin(\lambda) \triangleq \sup_{0 \le s < t \le \Tfin} \left|  \log \frac{\lambda(t)-\lambda(s)}{t-s}   \right| < \infty.
\end{equation*} 
For $\sfz_1$, $\sfz_2 \in D([0,\Tfin];\mathscr{Z})$, we now define
\begin{equation}\label{E:Sk_metric}
d_{[0,\Tfin]}(\sfz_1,\sfz_2) \triangleq \inf_{\lambda \in \Lambda_\Tfin} \left\{ \gamma_\Tfin(\lambda) \vee \sup_{0 \le t \le \Tfin} r\left(\sfz_1(t),\sfz_2(\lambda(t))\right)    \right\}.
\end{equation} 
Note that if $d^u_{[0,\Tfin]}(\sfz_1,\sfz_2) \triangleq \sup_{0 \le t \le \Tfin} r(\sfz_1(t),\sfz_2(t))$ is the uniform metric on $D([0,\Tfin];\mathscr{Z})$, then $d_{[0,\Tfin]}(\sfz_1,\sfz_2) \le d^u_{[0,\Tfin]}(\sfz_1,\sfz_2)$. Indeed, the latter corresponds to the specific choice $\lambda(t) \equiv t$.

We now state our main result.

\begin{theorem}[Main Theorem]\label{T:FLLN}
Fix $\mathfrak{T} \in \BN$, $0 \le \nu<2/3$. For $\eps \in (0,1)$, let $\Tfin_\eps \in \BN$ such that $\Tfin_\eps \le \mathfrak{T}/\eps^\nu$. Suppose  $z(\cdot)$, $Z^\eps_\cdot$ are given by  \eqref{E:det_state_full}, \eqref{E:stoch_state_full}, respectively, with initial conditions $z(0)=Z^\eps_0=(x^*,1)$, where $x^*$ is as in Proposition \ref{P:fixed_point}. Then, for any $p \in [1,\infty)$, we have that $d_{[0,\Tfin_\eps]}(Z^\eps,z) \to 0$ in $L^p$, i.e.,
\begin{equation}\label{E:conv-in-Lp}
\lim_{\eps \searrow 0} \BE \left[\left(d_{[0,\Tfin_\eps]}(Z^\eps,z)\right)^p \right]=0.
\end{equation}
%i.e., $d_{[0,\Tfin_\eps]}(Z^\eps,z) \to 0$ in $L^p$ as $\eps \searrow 0$.
\end{theorem}

\begin{remark}\label{R:conv-in-prob}
Of course, Theorem \ref{T:FLLN} implies that $d_{[0,\Tfin_\eps]}(Z^\eps,z)$ converges to $0$ in probability, i.e., for any $\vartheta>0$, we have
$\lim_{\eps \searrow 0} \BP \left\{ d_{[0,\Tfin_\eps]}(Z^\eps,z) \ge \vartheta \right\}=0$.
\end{remark}

To explain the intuition behind Theorem \ref{T:FLLN}, we note that when $\eps \ll 1$, the likely behaviour of $X^\eps_t$ is to closely track $x(t)$. Therefore, one expects that with high probability, we have $\tau_n^\eps \approx t_n$, $\sigma^\eps_n=s_n = n$ for each $1 \le n \le \Tfin_\eps$ (at least if $\Tfin_\eps$ is not too large). On this ``good" event, a random time-deformation $\lambda$, for which $\gamma_{\Tfin_\eps}(\lambda)$ is small, can be used to align the jumps of $Y^\eps_t$ and $y(t)$ so that $Y^\eps_{\lambda(t)} \equiv y(t)$. Continuity now ensures that $X^\eps_{\lambda(t)}$ is close to $x(t)$, and we get that $d_{[0,\Tfin_\eps]}(Z^\eps,z)$ can be bounded above by a term which goes to zero as $\eps \searrow 0$. It now remains to show that the probability of the complement of this event, i.e., the event where one or more of the $\tau_n^\eps$ differ from $t_n$ by a significant amount, is small. 

Our thoughts are organised as follows. First, we introduce an additional scale $\delta \searrow 0$ to quantify proximity of $\tau^\eps_n$ to $t_n$; we will later take $\delta = \eps^\varsigma$ for suitable $\varsigma>0$. Now, for $\eps,\delta \in (0,1)$, set $G_0^{\eps,\delta} \triangleq \Omega$, and for $n \ge 1$, define
\begin{equation*}
\begin{aligned}
G_n^{\eps,\delta} & \triangleq \{\omega \in G_{n-1}^{\eps,\delta}: |\tau_n^\eps(\omega)-t_n| \le \delta\}\\
B_n^{\eps,\delta} & \triangleq \{\omega \in G_{n-1}^{\eps,\delta}: |\tau_n^\eps(\omega)-t_n| > \delta\} = G_{n-1}^{\eps,\delta} \setminus G_n^{\eps,\delta}.
\end{aligned}
\end{equation*}
Note that the $G_n^{\eps,\delta}$'s are decreasing, i.e., $G_0^{\eps,\delta} \supset G_1^{\eps,\delta} \supset \dots$ and that the $B_n^{\eps,\delta}$'s are pairwise disjoint. Consequently, $ \cap_{n=1}^{\Tfin_\eps} G_n^{\eps,\delta}=G_{\Tfin_\eps}^{\eps,\delta} $ and $\BP \left( \cup_{n=1}^{\Tfin_\eps} B_n^{\eps,\delta} \right) = \sum_{n=1}^{\Tfin_\eps} \BP \left( B_n^{\eps,\delta} \right)$. 

We now outline the principal steps in proving Theorem \ref{T:FLLN}. First, in Proposition \ref{P:pth_power_estimate}, we derive a path-wise estimate for $d_{[0,\Tfin_\eps]}(Z^\eps,z)$ and its positive powers. This result assures us that our quantity of interest is indeed small on the event $G_{\Tfin_\eps}^{\eps,\delta}$ and of order $1$ on its complement. Then, in Proposition \ref{P:bn}, we obtain an upper bound on $\BP\left( B_n^{\eps,\delta} \right)$ in terms of the tail of the standard normal distribution. These two propositions enable us to complete the proof of Theorem \ref{T:FLLN}. Both Propositions \ref{P:pth_power_estimate} and \ref{P:bn} are proved through a series of Lemmas; the proofs of the latter are deferred to Section \ref{S:Proofs}.

We start by introducing some notation. Let $\mathsf{t}_\aon \triangleq t_n-s_{n-1}$ and $\mathsf{t}_\aoff \triangleq s_n-t_n$ denote the fractions of time in each interval $[n,n+1]$ for which the deterministic system is in the {\sc on} and {\sc off} states respectively, and let $\mathsf{t}_\sfmin \triangleq \mathsf{t}_\aon \wedge \mathsf{t}_\aoff$.

\begin{proposition}\label{P:pth_power_estimate}
For every $p>0$, there exists a constant $C_p>0$ such that for all $\eps,\delta \in (0,1)$, $\Tfin_\eps \in \BN$ satisfying $0 <\delta \le \mathsf{t}_\sfmin/(4\Tfin_\eps)$, we have
\begin{equation}\label{E:pth_power_estimate}
(d_{[0,\Tfin_\eps]}(Z^\eps,z))^p \le C_p \left\{ (\Tfin_\eps \delta)^p + \delta^p + 1_{\Omega\setminus G_{\Tfin_\eps}^{\eps,\delta}} + \eps^p \thinspace \Tfin_\eps^p \sup_{0 \le t \le \Tfin_\eps} |W_t|^p \right\}.
\end{equation}
\end{proposition}

To prove Proposition \ref{P:pth_power_estimate}, we will employ the (random) time-deformation $\lambda^\eps:\Omega \to \Lambda_{\Tfin_\eps}$ defined by 
\begin{multline}\label{E:time_deformation_a}
\lambda^\eps_t (\omega) \triangleq \sum_{n \ge 1} 1_{[s_{n-1},t_n)}(t) \cdot \left\{ s_{n-1} + \left(  \frac{\tau_n^\eps(\omega)-s_{n-1}}{t_n-s_{n-1}}\right) (t-s_{n-1}) \right\}\\ + \sum_{n \ge 1} 1_{[t_n,s_n)}(t) \cdot \left\{ \tau_n^\eps(\omega) + \left(  \frac{s_n-\tau_n^\eps(\omega)}{s_n-t_n}\right) (t-t_n) \right\} \qquad \text{if $\omega \in G_{\Tfin_\eps}^{\eps,\delta}$,}
\end{multline}
and
\begin{equation}\label{E:time_deformation_b}
\lambda^\eps_t (\omega) \triangleq t \qquad \text{if $\omega \in \Omega\setminus G_{\Tfin_\eps}^{\eps,\delta}$.}
\end{equation}
Note that in actuality, $\lambda^\eps=\lambda^{\eps,\delta}$. However, we have suppressed the $\delta$-dependence to reduce clutter and also because we will eventually take $\delta=\delta(\eps)$. 
The first step is to show that $\gamma_{\Tfin_\eps}(\lambda^\eps)$ is small; this is accomplished in Lemma \ref{L:time_deformation_norm} below. Next, in Lemma \ref{L:full_state_close}, we estimate $\sup_{0 \le t \le \Tfin_\eps} | X^\eps_{\lambda^\eps_t}-x(t) |^p$ and $\sup_{0 \le t \le \Tfin_\eps} | Y^\eps_{\lambda^\eps_t}-y(t) |^p$ for $p>0$. 

\begin{lemma}\label{L:time_deformation_norm}
Let $\Tfin_\eps \in \BN$. If $0 <\delta \le \mathsf{t}_\sfmin/(4\Tfin_\eps)$, then for each $\omega \in \Omega$, we have
\begin{equation}\label{E:time_deformation_norm}
\gamma_{\Tfin_\eps}(\lambda^\eps(\omega)) \le \frac{4\Tfin_\eps \delta}{\mathsf{t}_\sfmin}.
\end{equation}
\end{lemma}

\begin{lemma}\label{L:full_state_close}
For every $p>0$, there exists a constant $c_p>0$ such that for all $\eps,\delta \in (0,1)$, we have
\begin{equation}\label{E:full_state_close_ms}
\begin{aligned}
\sup_{0 \le t \le \Tfin_\eps} | X^\eps_{\lambda^\eps_t}-x(t) |^p & \le c_p \left\{ \delta^p\thinspace  1_{G_{\Tfin_\eps}^{\eps,\delta}} + 1_{\Omega\setminus G_{\Tfin_\eps}^{\eps,\delta}} + \eps^p \thinspace \Tfin_\eps^p \sup_{0 \le t \le \Tfin_\eps} |W_t|^p \right\},\\
\sup_{0 \le t \le \Tfin_\eps} |Y^\eps_{\lambda^\eps_t}-y(t)|^p & \le 1_{\Omega\setminus G_{\Tfin_\eps}^{\eps,\delta}}.
\end{aligned}
\end{equation}
\end{lemma}

Lemmas \ref{L:time_deformation_norm} and \ref{L:full_state_close} are proved in Section \ref{S:Proofs}. We now have

\begin{proof}[Proof of Proposition \ref{P:pth_power_estimate}]
Let $\lambda^\eps \in \Lambda_{\Tfin_\eps}$ be as in \eqref{E:time_deformation_a}, \eqref{E:time_deformation_b}. It is easily seen from \eqref{E:Sk_metric} that  \\
%\begin{equation*}
$(d_{[0,\Tfin_\eps]}(Z^\eps,z))^p \le 3^p \left\{ (\gamma_{\Tfin_\eps}(\lambda^\eps))^p + \sup_{0 \le t \le \Tfin_\eps} |X^\eps_{\lambda^\eps_t}-x(t)|^p + \sup_{0 \le t \le \Tfin_\eps} |Y^\eps_{\lambda^\eps_t}-y(t)|^p \right\}$.
%\end{equation*}
The claim \eqref{E:pth_power_estimate} now follows from Lemmas \ref{L:time_deformation_norm}, \ref{L:full_state_close}.
\end{proof}

We now estimate $\BP\left( B_n^{\eps,\delta} \right)$. Let $\normaltail$ denote the right tail of the normal distribution; i.e., 
\begin{equation*}
\normaltail(x) \triangleq \frac{1}{\sqrt{2 \pi}} \int_x^\infty e^{-t^2/2} dt \qquad \text{for $x \ge 0$.}
\end{equation*}
A simple integration by parts yields 
\begin{equation}\label{E:gaussian_tail_estimate_1}
\normaltail(x) \le \frac{3}{\sqrt{2\pi}} x^2 e^{-x^2/2},  \qquad \text{for $x \ge 1$.} 
\end{equation}

\begin{proposition}\label{P:bn}
There exists $\delta_0 \in (0,1)$ and $K>0$ such that whenever $0<\delta<\delta_0$, $\eps \in (0,1)$, and $n \ge 1$, we have
\begin{equation}\label{E:prob_bn}
\BP\left( B_n^{\eps,\delta} \right) \le 3 \normaltail \left( K \frac{\delta}{\eps} \right).
\end{equation}
\end{proposition}

To prove this proposition, we write the event $B_n^{\eps,\delta}$ as the disjoint union $B_n^{\eps,\delta,-}\cup B_n^{\eps,\delta,+}$ where
%\begin{equation*}
$B_n^{\eps,\delta,-} \triangleq \{\omega \in G_{n-1}^{\eps,\delta}: \tau_n^\eps(\omega) < t_n - \delta \}$ and $B_n^{\eps,\delta,+} \triangleq \{\omega \in G_{n-1}^{\eps,\delta}: \tau_n^\eps(\omega) > t_n + \delta \}$.
%\end{equation*}
The quantities $\BP( B_n^{\eps,\delta,-})$ and $\BP( B_n^{\eps,\delta,+})$ are now estimated separately in Lemmas \ref{L:bnminus} and \ref{L:bnplus} below, whose proofs are given in Section \ref{S:Proofs}.

\begin{lemma}\label{L:bnminus} 
There exists $K_->0$ such that for any $n \ge 1$, $\eps,\delta \in (0,1)$, we have
\begin{equation}\label{E:prob_bnminus}
\BP\left( B_n^{\eps,\delta,-} \right) \le 2 \normaltail \left( K_- \frac{\delta}{\eps} \right).
\end{equation}
\end{lemma}

\begin{lemma}\label{L:bnplus} 
There exists $\delta_+ \in (0,1)$ and $K_+>0$ such that whenever $0 < \delta < \delta_+$, $\eps \in (0,1)$, $n \ge 1$, we have
\begin{equation}\label{E:prob_bnplus}
\BP\left( B_n^{\eps,\delta,+} \right) \le \normaltail \left( K_+ \frac{\delta}{\eps} \right). 
\end{equation}
\end{lemma}
 
We now provide 
\begin{proof}[Proof of Proposition \ref{P:bn}]
From Lemmas \ref{L:bnminus} and \ref{L:bnplus}, we take $K \triangleq K_-\wedge K_+$, $\delta_0 \triangleq \delta_+$. Now, noting that $\normaltail$ is strictly decreasing, we use \eqref{E:prob_bnminus} and \eqref{E:prob_bnplus} to get \eqref{E:prob_bn}. 
\end{proof}

Finally, we have

\begin{proof}[Proof of Theorem \ref{T:FLLN}]
%We use here Proposition \ref{P:pth_power_estimate}. 
Fix $p \in [1,\infty)$ and let $\delta \triangleq \eps^\varsigma$ where $\nu<\varsigma<1$. By the Burkholder-Davis-Gundy inequalities \cite[Theorem 3.3.28]{KS91}, there exists a universal positive constant $k_{p/2}$ such that 
$\BE \left[ \sup_{0 \le t \le \Tfin_\eps} | W_t |^p  \right] \le k_{p/2} \Tfin_\eps^{p/2}$. Noting that $\BE[1_{\Omega\setminus G_{\Tfin_\eps}^{\eps,\delta}}]=\sum_{n=1}^{\Tfin_\eps} \BP(B_n^{\eps,\delta})$, we see from Propositions \ref{P:pth_power_estimate} and \ref{P:bn} that for $\eps \in (0,1)$ small enough,\footnote{One can check that $0 < \eps < \left(\frac{\mathsf{t}_\sfmin}{4\mathfrak{T}}\right)^{1/(\varsigma-\nu)}\wedge \delta_{0}^{1/\varsigma}$ will suffice.}
\begin{equation*}
\BE[(d_{[0,\Tfin_\eps]}(Z^\eps,z))^p] \le C_p \left\{ \mathfrak{T}^p \eps^{(\varsigma-\nu)p}  + \eps^{\varsigma p} +  \frac{3 \mathfrak{T}}{\eps^\nu} \normaltail \left(K \frac{1}{\eps^{1-\varsigma}}\right) + k_{p/2} \mathfrak{T}^{3p/2} \eps^{(1-3\nu/2)p}  \right\}.
\end{equation*}
Since $0 \le \nu<2/3$ and $\nu<\varsigma<1$, straightforward calculations using \eqref{E:gaussian_tail_estimate_1} yield \eqref{E:conv-in-Lp}.
\end{proof}

\section{Proofs of Lemmas}\label{S:Proofs}

%Recall that $G_{\Tfin_\eps}^{\eps,\delta}=\cap_{n=1}^{\Tfin_\eps} G_n^{\eps,\delta}$ and $\Omega\setminus G_{\Tfin_\eps}^{\eps,\delta}=\cup_{n=1}^{\Tfin_\eps} B_n^{\eps,\delta}$. 

\begin{proof}[Proof of Lemma \ref{L:time_deformation_norm}]
For $\omega \in \Omega\setminus G_{\Tfin_\eps}^{\eps,\delta}$, we have $\gamma_{\Tfin_\eps}(\lambda^\eps(\omega))=0$, implying \eqref{E:time_deformation_norm}. So, fix  $\omega \in G_{\Tfin_\eps}^{\eps,\delta}$. Note that the function $\lambda^\eps_\cdot(\omega)$ is piecewise-linear with ``corners" at $0=s_0<t_1<s_1<\dots<t_{\Tfin_\eps}<s_{\Tfin_\eps}=\Tfin_\eps$. Since $\omega \in G_{\Tfin_\eps}^{\eps,\delta}$, we have $\tau_n^\eps(\omega) \in [t_n-\delta,t_n+ \delta]$ for $1 \le n \le \Tfin_\eps$. Recalling that $\lambda^\eps_{t_n}(\omega) = \tau_n^\eps(\omega)$, it is now easy to see that
\begin{equation}\label{E:time_def_aux_1}
\max_{1 \le n \le \Tfin_\eps} \left\{ \left|  \frac{\lambda^\eps_{t_n}(\omega)-\lambda^\eps_{s_{n-1}}(\omega)}{t_n-s_{n-1}} - 1 \right| \vee \left|  \frac{\lambda^\eps_{s_n}(\omega)-\lambda^\eps_{t_n}(\omega)}{s_n-t_n} - 1 \right|  \right\} \le \frac{\delta}{\mathsf{t}_\sfmin}.
\end{equation}

Now, let $0 \le s < t \le \Tfin_\eps$ and let $\{u_0,u_1,\dots,u_k\}$ be a sequential enumeration of all corners starting just to the left of $s$ and ending just to the right of $t$, i.e., $u_0 \le s < u_1 < \dots < u_{k-1} < t \le u_k$. By the triangle inequality, we have
\begin{multline*}
|\lambda^\eps_t(\omega)-\lambda^\eps_s(\omega) - (t-s)| \le |t-u_{k-1}| \left| \frac{\lambda^\eps_t(\omega)-\lambda^\eps_{u_{k-1}}(\omega)}{t-u_{k-1}} - 1 \right| \\+ \sum_{i=2}^{k-1}  |u_i-u_{i-1}| \left| \frac{\lambda^\eps_{u_i}(\omega)-\lambda^\eps_{u_{i-1}}(\omega)}{u_i-u_{i-1}} - 1 \right| +  |u_1-s| \left| \frac{\lambda^\eps_{u_1}(\omega)-\lambda^\eps_s(\omega)}{u_1-s} - 1 \right|.
\end{multline*}
Noting that $|t-u_{k-1}|, |u_{k-1}-u_{k-2}|,\dots,|u_1-s|$ are less than $|t-s|$, recalling the piecewise-linear nature of $\lambda^\eps_t(\omega)$, and using \eqref{E:time_def_aux_1}, we get
\begin{equation*}
%\begin{aligned}
\left|  \frac{\lambda^\eps_t(\omega)-\lambda^\eps_s(\omega)}{t-s} - 1 \right| 
 \le \sum_{i=1}^{k} \left|  \frac{\lambda^\eps_{u_i}(\omega)-\lambda^\eps_{u_{i-1}}(\omega)}{u_i-u_{i-1}} - 1 \right| \le \frac{2\Tfin_\eps \delta}{\mathsf{t}_\sfmin}.
%\end{aligned}
\end{equation*}
Thus,
\begin{equation*}
\log \left( 1 - \frac{2\Tfin_\eps \delta}{\mathsf{t}_\sfmin} \right) \le \log \left( \frac{\lambda^\eps_t(\omega)-\lambda^\eps_s(\omega)}{t-s} \right) \le \log \left( 1+ \frac{2\Tfin_\eps \delta}{\mathsf{t}_\sfmin} \right). 
\end{equation*}
Using the estimate $|\log(1\pm x)| \le 2|x|$ for $|x| \le 1/2$ \cite[pp. 127]{ConvProbMeas}, we get that for $0<\delta \le \mathsf{t}_\sfmin/{4\Tfin_\eps}$, we have
\begin{equation*}
-\frac{4\Tfin_\eps \delta}{\mathsf{t}_\sfmin}  \le \log \left( \frac{\lambda^\eps_t(\omega)-\lambda^\eps_s(\omega)}{t-s} \right) \le \frac{4\Tfin_\eps\delta}{\mathsf{t}_\sfmin} .
\end{equation*}
Since $s,t$ are arbitrary, \eqref{E:time_deformation_norm} follows.
\end{proof}

\begin{proof}[Proof of Lemma \ref{L:full_state_close}]
We first bound $\sup_{0 \le t \le \Tfin_\eps} | X^\eps_{\lambda^\eps_t}-x(t) |^p$. 
Write 
%\begin{equation*}
$X^\eps_{\lambda^\eps_t}-x(t) = \sum_{i=1}^3 \mathsf{L}^{i,\eps}_t$
%\end{equation*}
where 
\begin{equation*}
\begin{aligned}
\mathsf{L}^{1,\eps}_t & \triangleq \sum_{n \ge 1} 1_{[\sigma_{n-1}^\eps,\tau_n^\eps)}(\lambda^\eps_t) \cdot \left\{  \frac{\beta}{\alpha_\aon} + \left(\sfX^{n-1,\eps,\aoff}_{\sigma_{n-1}^\eps}- \frac{\beta}{\alpha_\aon}\right) e^{-\alpha_\aon (\lambda^\eps_t-\sigma_{n-1}^\eps)}   \right\}\\
& \phantom{\triangleq} - \sum_{n \ge 1} 1_{[s_{n-1},t_n)}(t) \cdot \left\{\frac{\beta}{\alpha_\aon} + \left(\sfx^{n-1,\aoff}(s_{n-1}) - \frac{\beta}{\alpha_\aon}\right) e^{-\alpha_\aon (t-s_{n-1})}\right\},\\
\mathsf{L}^{2,\eps}_t & \triangleq \sum_{n \ge 1} 1_{[\tau_{n}^\eps,\sigma^\eps_n)}(\lambda^\eps_t) \cdot x_\refer \medspace e^{-\alpha_\aoff(\lambda^\eps_t-\tau_{n}^\eps)} - \sum_{n \ge 1} 1_{[t_{n},s_n)}(t) \cdot x_\refer \medspace e^{-\alpha_\aoff(t-t_{n})}, \\
\mathsf{L}^{3,\eps}_t & \triangleq \sum_{n \ge 1} 1_{[\sigma_{n-1}^\eps,\tau_n^\eps)}(\lambda^\eps_t) \cdot \eps \thinspace e^{-\alpha_\aon \lambda^\eps_t} \left(I_{\lambda^\eps_t} - I_{\lambda^\eps_t \wedge \sigma_{n-1}^\eps}  \right). 
\end{aligned}
\end{equation*}
We start by noting that $\mathsf{L}^{1,\eps}_t(\omega)$, $\mathsf{L}^{2,\eps}_t(\omega)$ are bounded for all $(t,\omega) \in [0,\Tfin_\eps]\times\Omega$. We will now show that for $\omega \in G_{\Tfin_\eps}^{\eps,\delta}$, $|\mathsf{L}^{1,\eps}_t|$, $|\mathsf{L}^{2,\eps}_t|$ are in fact of order $\delta$. 
We note that if $\omega \in G_{\Tfin_\eps}^{\eps,\delta}$, then $\lambda^\eps_t(\omega) \in [\sigma^\eps_{n-1}(\omega),\tau^\eps_n(\omega))$ iff $t \in [s_{n-1},t_n)$ for all $1 \le n \le \Tfin_\eps$. Thus, we have for each $t \in [0,\Tfin_\eps]$,
\begin{equation*}
\begin{aligned}
1_{G_{\Tfin_\eps}^{\eps,\delta}}\thinspace \cdot \mathsf{L}^{1,\eps}_t &= 1_{G_{\Tfin_\eps}^{\eps,\delta}} \cdot \sum_{n \ge 1} 1_{[s_{n-1},t_n)}(t) \cdot \left\{ \left(\sfX^{n-1,\eps,\aoff}_{\sigma_{n-1}^\eps}- \frac{\beta}{\alpha_\aon}\right) e^{-\alpha_\aon (\lambda^\eps_t-\sigma_{n-1}^\eps)}\right.\\ & \qquad \qquad \qquad \qquad \qquad \qquad \qquad \qquad \left. - \left(\sfx^{n-1,\aoff}(s_{n-1}) - \frac{\beta}{\alpha_\aon}\right) e^{-\alpha_\aon (t-s_{n-1})} \right\}\\
&= 1_{G_{\Tfin_\eps}^{\eps,\delta}} \cdot \sum_{n \ge 1} 1_{[s_{n-1},t_n)}(t) \cdot \left\{ \left(\sfX^{n-1,\eps,\aoff}_{\sigma_{n-1}^\eps} - \sfx^{n-1,\aoff}(s_{n-1})\right) e^{-\alpha_\aon(\lambda^\eps_t-\sigma^\eps_{n-1})} \right.\\
&\left. \qquad \qquad \qquad \qquad \qquad + \left(\sfx^{n-1,\aoff}(s_{n-1})-\frac{\beta}{\alpha_\aon}\right) \left(e^{-\alpha_\aon(\lambda^\eps_t-\sigma^\eps_{n-1})} - e^{-\alpha_\aon(t-s_{n-1})}\right) \right\}.
\end{aligned}
\end{equation*}
It is now easy to see that 
\begin{equation*}
%\begin{aligned}
\left| 1_{G_{\Tfin_\eps}^{\eps,\delta}}\thinspace \cdot  \mathsf{L}^{1,\eps}_t \right|  \le 1_{G_{\Tfin_\eps}^{\eps,\delta}} \cdot  \sum_{n \ge 1} 1_{[s_{n-1},t_n)}(t) \cdot \left\{ x_\refer \thinspace \alpha_\aoff \thinspace \delta + \left(\frac{\beta}{\alpha_\aon}-x^*\right) \alpha_\aon \delta\right\}.
%\end{aligned}
\end{equation*}
Hence, there exists $K_1>0$ such that for all $t \in [0,\Tfin_\eps]$, 
\begin{equation}\label{E:cont_state_close_ms_1}
|\mathsf{L}^{1,\eps}_t(\omega)| \le 1_{G_{\Tfin_\eps}^{\eps,\delta}}(\omega) K_1 \delta + 1_{\Omega\setminus G_{\Tfin_\eps}^{\eps,\delta}}(\omega) K_1.
\end{equation}
Turning to $\mathsf{L}^{2,\eps}_t$, we note that
\begin{equation*}
%\begin{aligned}
1_{G_{\Tfin_\eps}^{\eps,\delta}}\thinspace \cdot \mathsf{L}^{2,\eps}_t = 1_{G_{\Tfin_\eps}^{\eps,\delta}}\thinspace \cdot \sum_{n \ge 1} 1_{[t_{n},s_n)}(t)\cdot x_\refer \medspace \left\{ e^{-\alpha_\aoff(\lambda^\eps_t-\tau_{n}^\eps)} - e^{-\alpha_\aoff(t-t_{n})}\right\},
%\end{aligned}
\end{equation*}
which gives 
\begin{equation*}
\left|1_{G_{\Tfin_\eps}^{\eps,\delta}}\thinspace \cdot \mathsf{L}^{2,\eps}_t \right| \le 1_{G_{\Tfin_\eps}^{\eps,\delta}}\thinspace \cdot \sum_{n \ge 1} 1_{[t_{n},s_n)}(t) \cdot x_\refer \thinspace \alpha_\aoff \thinspace |\lambda^\eps_t-\tau^\eps_n-t+t_n| \le 2 x_\refer \thinspace \alpha_\aoff \thinspace \delta. 
\end{equation*}
Hence, there exists $K_2>0$ such that for all $t \in [0,\Tfin_\eps]$, 
\begin{equation}\label{E:cont_state_close_ms_2}
|\mathsf{L}^{2,\eps}_t(\omega)| \le 1_{G_{\Tfin_\eps}^{\eps,\delta}}(\omega) K_2 \delta + 1_{\Omega\setminus G_{\Tfin_\eps}^{\eps,\delta}}(\omega) K_2.
\end{equation}
Turning now to $\mathsf{L}^{3,\eps}_t$, we use integration by parts to get $I_t = e^{\alpha_\aon t} W_t - \alpha_\aon \int_0^t W_s e^{\alpha_\aon s}ds$. It now follows that for any $t \in [0,\Tfin_\eps]$,
\begin{equation*}
\mathsf{L}^{3,\eps}_t = \eps \thinspace \sum_{n \ge 1} 1_{[\sigma_{n-1}^\eps,\tau_n^\eps)}(\lambda^\eps_t) \cdot\left[ W_{\lambda^\eps_t} - e^{-\alpha_\aon [\lambda^\eps_t - \lambda^\eps_t \wedge \sigma_{n-1}^\eps]} W_{\lambda^\eps_t \wedge \sigma_{n-1}^\eps} -\alpha_\aon e^{-\alpha_\aon \lambda_t^\eps} \int_{\lambda_t^\eps\wedge\sigma^\eps_{n-1}}^{\lambda^\eps_t} W_s e^{\alpha_\aon s} ds \right],
\end{equation*} 
whence
%\begin{equation*}
%\left( \mathsf{L}^{3,\eps}_t \right)^2 \le 9 \eps^2 \left\{ 2 \sup_{0 \le t \le \Tfin_\eps} W_t^2 + \alpha_\aon^2 \Tfin_\eps^2 \sup_{0 \le t \le \Tfin_\eps} W_t^2 \right\}
%\end{equation*}
%and
%\begin{equation*}
$|\mathsf{L}^{3,\eps}_t | \le \eps \left( 2 + \alpha_\aon\right) \Tfin_\eps \sup_{0 \le t \le \Tfin_\eps} |W_t|$. Recalling \eqref{E:cont_state_close_ms_1} and \eqref{E:cont_state_close_ms_2}, we see that for $p>0$, there exists $c_p>0$ such that the first line of \eqref{E:full_state_close_ms} holds.

To bound $\sup_{0 \le t \le \Tfin_\eps} | Y^\eps_{\lambda^\eps_t}-y(t) |^p$, note that for $\omega \in G_{\Tfin_\eps}^{\eps,\delta}$, we have $Y^\eps_{\lambda^\eps_t}=y(t)$ for all $t \in [0,\Tfin_\eps]$. Consequently,
$\sup_{0 \le t \le \Tfin_\eps} |Y^\eps_{\lambda^\eps_t}-y(t)|^p = 1_{\Omega\setminus G_{\Tfin_\eps}^{\eps,\delta}} \cdot \sup_{0 \le t \le \Tfin_\eps} |Y^\eps_{\lambda^\eps_t}-y(t)|^p \le 1_{\Omega\setminus G_{\Tfin_\eps}^{\eps,\delta}}$. 
\end{proof}

%\newpage

To state and prove Lemmas \ref{L:bnminus} and \ref{L:bnplus}, we will need some notation. For $n \in \BN$, $\xi \in (0,x_\refer)$, we let
%\begin{equation}\label{E:an_defn}
%\begin{aligned}
$a_{n}(t;\xi) \triangleq 1_{[n-1,\infty)}(t) \cdot \left\{\beta/\alpha_\aon + \left( \xi -  \beta/\alpha_\aon  \right) e^{-\alpha_\aon(t-(n-1))}\right\}$.
%\sfM^{n}_t & \triangleq 1_{[n-1,\infty)}(t) \cdot  \int_{n-1}^t e^{-\alpha_\aon(t-u)}dW_u.
%\end{aligned}
%\end{equation}
It is now easily checked that for $t \ge n-1$ and $\varkappa \in (0,1)$ small enough,
\begin{equation}\label{E:an_envelope}
a_n(t;x^*)-\varkappa \le a_n(t;x^*-\varkappa) \le a_n(t;x^*) \le a_n(t;x^* + \varkappa) \le a_n(t;x^*) + \varkappa. 
\end{equation}

We will also find it helpful to express the continuous square-integrable martingale $I_t = \int_0^t e^{\alpha_\aon u}dW_u$ as a time-changed Brownian motion. The quadratic variation process of $I$ given by
%\begin{equation*}
$\la I \ra_t = \int_0^t e^{2 \alpha_\aon u} du = (e^{2 \alpha_\aon t} - 1)/(2 \alpha_\aon)$ for $t \ge 0$
%\end{equation*}
is strictly increasing with $\lim_{t \to \infty} \la I \ra_t=\infty$. It therefore follows \cite[Theorem 3.4.6]{KS91} that the process $V = \{ V_s, \mathscr{G}_s:0 \le s < \infty\}$ defined by
%\begin{equation*}
$V_s \triangleq I_{T(s)}$, $\mathscr{G}_s \triangleq \filt_{T(s)}$ where $T(s) \triangleq \inf\{t \ge 0:\la I \ra_t > s\}= [\log \left( 1 + 2 \alpha_\aon s  \right)]/(2 \alpha_\aon)$,
%\end{equation*}
is a standard one-dimensional Brownian motion, and further, that 
%\begin{equation*}
$I_t = V_{\la I \ra_t}$ for all $t \ge 0$.
%\end{equation*}

Below, we will use the fact that if $\omega \in G_{n-1}^{\eps,\delta}$ (where $n \in \BN$), then $\sigma_{n-1}^\eps=n-1$, and
%\begin{equation*}
$|\sfX^{n-1,\eps,\aoff}_{\sigma_{n-1}^\eps}(\omega)-x^*| \le x_\refer \thinspace \alpha_\aoff \thinspace |\tau_{n-1}^\eps(\omega)-t_{n-1}| \le x_\refer \thinspace \alpha_\aoff \thinspace \delta$.
%\end{equation*}
Set
\begin{equation}\label{E:mu_definition}
\mu \triangleq \beta - (\alpha_\aon + \alpha_\aoff) x_\refer, \qquad \varkappa \triangleq x_\refer \thinspace \alpha_\aoff \thinspace \delta.
\end{equation}
Note that, on account of the upper bound on $\alpha_\aoff$ in \eqref{E:alpha_off_assumption_real}, we have $\mu>0$.

\begin{proof}[Proof of Lemma \ref{L:bnminus}]
We start by noting that for $n \ge 1$, 
\begin{equation*}
\sfX^{n,\eps,\aon}_t \cdot 1_{G_{n-1}^{\eps,\delta}} (\omega) = 1_{G_{n-1}^{\eps,\delta}} (\omega) \cdot 1_{[n-1,\infty)}(t) \cdot \left\{ a_n\left( t; \sfX^{n-1,\eps,\aoff}_{\sigma_{n-1}^\eps} \right) + \eps e^{-\alpha_\aon t} \left( I_t - I_{n-1} \right) \right\}.
\end{equation*}
Using the fact that for $\omega \in G_{n-1}^{\eps,\delta}$, we have $\sfX^{n-1,\eps,\aoff}_{\sigma_{n-1}^\eps}(\omega) \in [x^*-\varkappa,x^*+\varkappa]$, together with \eqref{E:an_envelope}, we get
\begin{multline*}
B_n^{\eps,\delta,-} \subset \left\{ \omega \in G_{n-1}^{\eps,\delta}: \sup_{t \in [n-1,t_n-\delta]} \left( a_n\left( t; \sfX^{n-1,\eps,\aoff}_{\sigma_{n-1}^\eps} \right) + \eps e^{-\alpha_\aon t} \left( I_t - I_{n-1} \right)  \right) \ge x_\refer \right\}\\ \subset \left\{ \omega \in G_{n-1}^{\eps,\delta}:  \sup_{t \in [n-1,t_n-\delta]}   e^{-\alpha_\aon t} \left( I_t - I_{n-1} \right) \ge \frac{x_\refer - a_n(t_n-\delta; x^*) - \varkappa }{\eps} \right\}\\ \subset \left\{ \omega \in \Omega:  \sup_{t \in [n-1,t_n-\delta]}   e^{-\alpha_\aon t} \left( I_t - I_{n-1} \right) \ge \frac{\mu \delta}{\eps} \right\}
\end{multline*}
where the latter set inclusion uses the fact that $x_\refer - a_n(t_n-\delta;x^*) - \varkappa \ge \delta \mu$. We now easily get that
\begin{equation*}
\BP \left( B_n^{\eps,\delta,-} \right) \le \BP \left\{ \sup_{t \in [n-1,t_n-\delta]} \left( V_{\la I \ra_t} - V_{\la I \ra_{n-1}} \right) \ge \frac{\mu \delta e^{\alpha_\aon(n-1)}}{\eps} \right\}.
\end{equation*}
Letting
%\begin{equation*}
$u_n \triangleq \la I \ra_{n-1}$, $q \triangleq  \la I \ra_t-u_n$, $v_n \triangleq \la I \ra_{t_n-\delta}$, 
%\end{equation*}
and noting that $\hat{V}_q \triangleq V_{u_n + q} - V_{u_n}$ is a Brownian motion, we get 
\begin{equation*}
\BP \left( B_n^{\eps,\delta,-} \right) \le \BP \left\{ \sup_{q \in [0,v_n-u_n]} \hat{V}_q \ge \frac{\mu \delta e^{\alpha_\aon(n-1)}}{\eps} \right\} = 2\normaltail \left( \frac{\mu \delta}{\eps \sqrt{ \frac{e^{2 \alpha_\aon(t^*-\delta)} - 1} {2 \alpha_\aon} }}  \right), 
\end{equation*}
where we have explicitly computed $u_n$, $v_n$, and also used \cite[Remark 2.8.3]{KS91}. Since $e^{2\alpha_\aon(t^*-\delta)}-1 \le e^{2 \alpha_\aon t^*}$, we easily get \eqref{E:prob_bnminus} with $K_- \triangleq \sqrt{2 \alpha_\aon} e^{-\alpha_\aon t^*} \mu$.
\end{proof}

\begin{proof} [Proof of Lemma \ref{L:bnplus}]
Using the fact that for $\omega \in G_{n-1}^{\eps,\delta}$, we have $\sfX^{n-1,\eps,\aoff}_{\sigma_{n-1}^\eps}(\omega) \in [x^*-\varkappa,x^*+\varkappa]$, together with \eqref{E:an_envelope}, we get
\begin{multline*}
B_n^{\eps,\delta,+} = \left\{\omega \in G_{n-1}^{\eps,\delta}: \sup_{t \in [n-1,t_n+\delta]} \left( a_n(t;\sfX^{n-1,\eps,\aoff}_{\sigma_{n-1}^\eps}) + \eps e^{-\alpha_\aon t} (I_t - I_{n-1}) \right) < x_\refer \right\} \\ \subset \left\{  \omega \in G_{n-1}^{\eps,\delta}: a_n(t_n+\delta;x^*) - \varkappa + \eps e^{-\alpha_\aon (t_n + \delta)} \left(I_{t_n+\delta}-I_{n-1}\right) < x_\refer \right\} \\ \subset \left\{ \omega \in \Omega: e^{-\alpha_\aon (t_n + \delta)} \left(I_{t_n+\delta}-I_{n-1}\right) <  \left( \frac{x_\refer+\varkappa-a_n(t_n+\delta;x^*)}{\eps} \right) \right\}
\end{multline*}
Recalling Assumption \ref{A:assumptions_real}, a bit of computation reveals that if we let
\begin{equation*}\label{E:delta_zero}
\delta_+ \triangleq \frac{1}{\alpha_\aon} \log \left[ \frac{2 \beta - 2 \alpha_\aon x_\refer}{\beta - \alpha_\aon x_\refer + \alpha_\aoff x_\refer} \right] > 0,
\end{equation*}
then, for $0 < \delta <\delta_+$, we have
\begin{equation*}
\frac{x_\refer+\varkappa-a_n(t_n+\delta;x^*)}{\eps} < -\left( \frac{\mu}{2} \right) \frac{\delta}{\eps} < 0
\end{equation*}
Since $e^{-\alpha_\aon(t_n+\delta)} (I_{t_n+\delta}-I_{n-1}) \sim \mathscr{N} \left(0, \frac{1-e^{-2\alpha_\aon(t^*+\delta)}}{2 \alpha_\aon} \right)$, a straightforward calculation yields that for $0<\delta<\delta_+$, \eqref{E:prob_bnplus} holds with $K_+ \triangleq \mu\sqrt{\alpha_\aon/2}$. 
\end{proof}

\end{document}